\documentclass[reqno]{amsart}
\usepackage{pictexwd,dcpic}
\usepackage{slashbox}

\newtheorem{theorem}{Theorem}[section]
\newtheorem{proposition}[theorem]{Proposition}
\newtheorem{lemma}[theorem]{Lemma}
\newtheorem{corollary}[theorem]{Corollary}

\theoremstyle{definition}

\newtheorem{example}[theorem]{Example}

\theoremstyle{remark}
\newtheorem{remark}[theorem]{Remark}

\numberwithin{equation}{section}

\begin{document}

\title{A modularity criterion for Klein forms, with an application to modular forms of level $13$}

\author{Ick Sun Eum}
\address{Department of Mathematical Sciences, KAIST}
\curraddr{Daejeon 373-1, Korea} \email{zandc@kaist.ac.kr}
\thanks{}

\author{Ja Kyung Koo}
\address{Department of Mathematical Sciences, KAIST}
\curraddr{Daejeon 373-1, Korea} \email{jkkoo@math.kaist.ac.kr}
\thanks{}

\author{Dong Hwa Shin}
\address{Department of Mathematical Sciences, KAIST}
\curraddr{Daejeon 373-1, Korea} \email{shakur01@kaist.ac.kr}
\thanks{}


\keywords{modular forms, Klein forms, Dedekind eta-function.
\newline The research was partially supported by Basic Science Research Program through the NRF of Korea
funded by MEST (2010-0001654).}

\begin{abstract}
We find some modularity criterion for a product of Klein forms of
the congruence subgroup $\Gamma_1(N)$ (Theorem \ref{modularity2})
and, as its application, construct a basis of the space of modular
forms for $\Gamma_1(13)$ of weight $2$ (Example
\ref{fourvariables}). In the process we face with an interesting
property about the coefficients of certain theta function from a
quadratic form and prove it conditionally by applying Hecke
operators (Proposition \ref{proof}).
\end{abstract}

\maketitle

\section{Introduction}

The \textit{Dedekind eta-function} $\eta(\tau)$ is defined to be the
infinite product
\begin{equation}\label{eta}
\eta(\tau)=q^\frac{1}{24}\prod_{n=1}^\infty(1-q^n)\quad(\tau\in\mathfrak{H})
\end{equation}
where $q=e^{2\pi i\tau}$ and
$\mathfrak{H}=\{\tau\in\mathbb{C}:\mathrm{Im}(\tau)>0\}$. This
function plays an important role of building block which constitutes
various modular forms of integral or half-integral weight. For
example, the classical theta function
\begin{equation*}
\Theta(\tau)=\sum_{n=-\infty}^\infty
q^{n^2}\quad(\tau\in\mathfrak{H}),
\end{equation*}
which is a modular form for $\Gamma_0(4)$ of weight $1/2$
(\cite{Hecke}), can be written as
\begin{equation}\label{firsttheta}
\Theta(\tau)=\prod_{n=1}^\infty
\frac{(1-q^{2n})^5}{(1-q^n)^2(1-q^{4n})^2}
=\frac{\eta(2\tau)^5}{\eta(\tau)^2\eta(4\tau)^2}
\end{equation}
by the Jacobi's Triple Product Identity (\cite{Fine} $\S$17)
\begin{equation}\label{Jacobi}
\prod_{n=1}^\infty(1-q^{2n})(1+aq^{2n-1})(1+a^{-1}x^{2n-1})=
\sum_{m=-\infty}^\infty a^{m}q^{m^2}.
\end{equation}
And, every modular form for $\mathrm{SL}_2(\mathbb{Z})$ is known to
be expressed as a rational function in $\eta(\tau)^8$,
$\eta(2\tau)^4$ and $\eta(4\tau)^8$ (\cite{Ono} Theorem 1.67).
\par
On the other hand, we are further required to present more building
blocks to construct modular forms of integral weight for modular
groups of higher level. To this end we focus on the following Klein
forms.
\par
For $(r_1,r_2)\in\mathbb{Q}^2-\mathbb{Z}^2$ the \textit{Klein form}
$\mathfrak{k}_{(r_1,r_2)}(\tau)$ is defined by the following
infinite product expansion
\begin{equation}\label{KleinForm}
\mathfrak{k}_{(r_1,r_2)}(\tau)=e^{\pi
ir_2(r_1-1)}q^{\frac{1}{2}r_1(r_1-1)}(1-q_z)
\prod_{n=1}^\infty(1-q^nq_z)(1-q^nq_z^{-1})(1-q^n)^{-2}\quad(\tau\in\mathfrak{H})
\end{equation}
where $q_z=e^{2\pi iz}$ with $z=r_1\tau+r_2$. We see from Example
\ref{examplehalf} that the Klein forms seem to be a variation of
$\eta(\tau)^{-2}$. (In the original definition (\cite{K-L} Chapter 2
$\S$1) there is an extra factor $i/2\pi$.) Furthermore, we know
directly from the definition that it is a holomorphic function which
has no zeros and poles on $\mathfrak{H}$.
\par
In this paper we shall investigate some modularity criterions for
products of Klein forms of modular groups $\Gamma_1(N)$ of arbitrary
level (Theorems \ref{modularity2} and \ref{etaquotient}). As
applications we shall express theta functions associated with
quadratic forms in view of Klein forms and find a basis of the space
of modular forms for $\Gamma_1(13)$ of weight $2$ (Examples
\ref{twovariables} and \ref{fourvariables}).
\par
Let $\Theta_Q(\tau)=\sum_{n=0}^\infty r_Q(n)q^n$ be the theta
function associated with the quadratic form
$Q(x_1,x_2,x_3,x_4)=x_1^2+2x_2^2+x_3^2+x_4^2+x_1x_3+x_1x_4+x_2x_4$
where $r_Q(n)$ is the cardinality of the solution set
$\{\mathbf{x}\in\mathbb{Z}^4:Q(\mathbf{x})=n\}$ for $n\geq0$. We
shall find some primes $p$ which satisfy an interesting relation
\begin{equation*}
r_Q(p^2n)=\frac{r_Q(p^2)r_Q(n)}{r_Q(1)}\quad\textrm{for any integer
$n\geq1$ prime to $p$}
\end{equation*}
by applying Hecke operators to $\Theta_Q(\tau)$ (Proposition
\ref{proof} and Remark \ref{lastrmk}).
\par
Cho-Kim-Koo recently performed in \cite{C-K-K} a similar work about
modularity of Klein forms and constructed bases of certain spaces of
modular forms by describing the Fourier coefficients of some finite
products of Klein forms in terms of divisor functions. For the
purpose they adopted some useful nine identities between the
$q$-products and the $q$-series from the basic hypergeometric series
(\cite{Fine}). Thus, due to this technical restriction they could
hardly find examples of higher level, from which our work was
motivated to improve modularity criterion for $\Gamma_1(N)$.

\section{Modularity criterions}

First, we start with recalling some necessary transformation
formulas investigated in \cite{K-L}.

\begin{proposition}\label{TransformKlein}
\begin{itemize}
\item[(i)] For $(r_1,r_2)\in\mathbb{Q}^2-\mathbb{Z}^2$ and
$(s_1, s_2)\in\mathbb{Z}^2$ we get
\begin{eqnarray*}
\mathfrak{k}_{(-r_1,-r_2)}(\tau)&=&-\mathfrak{k}_{(r_1,r_2)}(\tau)\\
\mathfrak{k}_{(r_1,r_2)+(s_1,s_2)}(\tau)&=&(-1)^{s_1s_2+s_1+s_2}e^{-\pi
i(s_1r_2-s_2r_1)}\mathfrak{k}_{(r_1,r_2)}(\tau).
\end{eqnarray*}
\item[(ii)] For $(r_1,r_2)\in\mathbb{Q}^2-\mathbb{Z}^2$ and
$\alpha=\begin{pmatrix}a&b\\c&d\end{pmatrix}\in\mathrm{SL}_2(\mathbb{Z})$
we derive
\begin{equation*}
\mathfrak{k}_{(r_1,r_2)}(\tau)\circ\alpha=
\mathfrak{k}_{(r_1,r_2)}\bigg(\frac{a\tau+b}{c\tau+d}\bigg)=
(c\tau+d)^{-1}\mathfrak{k}_{(r_1,r_2)\alpha}(\tau) =
(c\tau+d)^{-1}\mathfrak{k}_{(r_1a+r_2c,r_1b+r_2d)}(\tau).
\end{equation*}
\item[(iii)] Let $\mathbf{B}_2(X)=X^2-X+1/6$ be the second Bernoulli polynomial and $\langle X\rangle$ be the fractional part of
$X\in\mathbb{R}$ so that $0\leq\langle X\rangle<1$. For
$(r_1,r_2)\in\mathbb{Q}^2-\mathbb{Z}^2$ we have
\begin{equation*}
\mathrm{ord}_{q}~\mathfrak{k}_{(r_1,r_2)}(\tau)=\frac{1}{2}\bigg(\mathbf{B}_2(\langle
r_1\rangle)-\frac{1}{6}\bigg)=\frac{1}{2}\langle r_1\rangle(\langle
r_1\rangle-1).
\end{equation*}
\end{itemize}
\end{proposition}
\begin{proof}
See \cite{K-L} Chapter 2 $\S$1.
\end{proof}

For every integer $k$,
$\alpha=\begin{pmatrix}a&b\\c&d\end{pmatrix}\in\mathrm{SL}_2(\mathbb{Z})$
and a function $f(\tau)$ on $\mathfrak{H}$ we write
\begin{equation*}
f(\tau)|[\alpha]_k=(c\tau+d)^{-k}(f(\tau)\circ\alpha).
\end{equation*}
And, we mainly consider the following three congruence subgroups
\begin{eqnarray*}
\Gamma(N)&=&\bigg\{\alpha\in\mathrm{SL}_2(\mathbb{Z})~:~
\alpha\equiv\left(\begin{matrix}1&0\\0&1\end{matrix}\right)\pmod{N}\bigg\}\\
\Gamma_1(N)&=&\bigg\{\alpha\in\mathrm{SL}_2(\mathbb{Z})~:~
\alpha\equiv\left(\begin{matrix}1&*\\0&1\end{matrix}\right)\pmod{N}\bigg\}\\
\Gamma_0(N)&=&\bigg\{\alpha\in\mathrm{SL}_2(\mathbb{Z})~:~
\alpha\equiv\left(\begin{matrix}*&*\\0&*\end{matrix}\right)\pmod{N}\bigg\}
\end{eqnarray*}
for an integer $N\geq2$. When $\Gamma$ is one of the above
congruence subgroups and $k$ is any integer, we say that a
holomorphic function $f(\tau)$ on $\mathfrak{H}$ is a
\textit{modular form for $\Gamma$ of weight $k$} if
\begin{itemize}
\item[(i)] $f(\tau)|[\gamma]_k=f(\tau)$ for all $\gamma\in\Gamma$;
\item[(ii)] $f(\tau)$ is holomorphic at every cusp (\cite{Shimura}
Definition 2.1).
\end{itemize}
We denote by $M_k(\Gamma)$ the $\mathbb{C}$-vector space of modular
forms for $\Gamma$ of weight $k$. If we replace (ii) by
\begin{itemize}
\item[(ii)$^\prime$] $f(\tau)$ is meromorphic at every cusp,
\end{itemize}
then we call $f(\tau)$ a \textit{nearly holomorphic modular form for
$\Gamma$ of weight $k$}.
\par
Kubert and Lang (\cite{K-L}) gave the following modularity condition
for $\Gamma(N)$.

\begin{proposition}\label{modularity}
For an integer $N\geq2$, let
$\{m(r)\}_{r\in\frac{1}{N}\mathbb{Z}^2-\mathbb{Z}^2}$ be a family of
integers such that $m(r)=0$ except finitely many $r$. Then the
product of Klein forms
\begin{equation*}
\prod_{r=(r_1,r_2)\in\frac{1}{N}\mathbb{Z}^2-\mathbb{Z}^2}
\mathfrak{k}_r(\tau)^{m(r)}
\end{equation*}
is a nearly holomorphic modular form for $\Gamma(N)$ of weight
$-\sum_r m(r)$ if and only if
\begin{eqnarray*}
&&\sum_r m(r)(Nr_1)^2\equiv\sum_r
m(r)(Nr_2)^2\equiv0\pmod{\gcd(2,N)\cdot
N}\\
&&\sum_r m(r)(Nr_1)(Nr_2)\equiv0\pmod{N}.
\end{eqnarray*}
\end{proposition}
\begin{proof}
See \cite{K-L} Chapter 3 Theorem 4.1.
\end{proof}

\begin{remark}\label{firstremark}
Let $N\geq2$ and $r\in\frac{1}{N}\mathbb{Z}^2-\mathbb{Z}^2$. Then
$\mathfrak{k}_r(\tau)$ (respectively, $\mathfrak{k}_r(\tau)^{2N}$)
is a nearly holomorphic modular form for $\Gamma(2N^2)$
(respectively, $\Gamma(N)$) of weight $-1$ (respectively, $-2N$).
\end{remark}

Now we shall develop a modularity criterion for the congruence
subgroup $\Gamma_1(N)$.

\begin{lemma}\label{distribution}
For an integer $N\geq2$ let $t$ be an integer with
$t\not\equiv0\pmod{N}$. Then we have the relation
\begin{equation*}
\mathfrak{k}_{(\frac{t}{N},0)}(N\tau)= Ne^{\frac{\pi
it}{2}(\frac{1}{N}-1)}
\prod_{n=0}^{N-1}\mathfrak{k}_{(\frac{t}{N},\frac{n}{N})}(\tau)
\prod_{n=1}^{N-1}\mathfrak{k}_{(0,\frac{n}{N})}(\tau)^{-1}.
\end{equation*}
\end{lemma}
\begin{proof}
By the definition (\ref{KleinForm}) we have
\begin{equation*}
\mathfrak{k}_{(\frac{t}{N},0)}(N\tau)=q^{\frac{t}{2}(\frac{t}{N}-1)}
(1-q^t)\prod_{m=1}^\infty(1-q^{Nm+t})(1-q^{Nm-t})(1-q^{Nm})^{-2},
\end{equation*}
and
\begin{eqnarray*}
&&\prod_{n=0}^{N-1}\mathfrak{k}_{(\frac{t}{N},\frac{n}{N})}(\tau)
\prod_{n=1}^{N-1}\mathfrak{k}_{(0,\frac{n}{N})}(\tau)^{-1}\\
&=&\prod_{n=0}^{N-1}\bigg(e^{\frac{\pi in}{N}(\frac{t}{N}-1)}
q^{\frac{t}{2N}(\frac{t}{N}-1)}(1-e^{\frac{2\pi
in}{N}}q^\frac{t}{N})\prod_{m=1}^{\infty} (1-e^\frac{2\pi
in}{N}q^{m+\frac{t}{N}}) (1-e^{-\frac{2\pi
in}{N}}q^{m-\frac{t}{N}})(1-q^m)^{-2}\bigg)\\
&&\times \prod_{n=1}^{N-1}\bigg( e^{-\frac{\pi
in}{N}}(1-e^\frac{2\pi in}{N})\prod_{m=1}^\infty(1-e^\frac{2\pi
in}{N}q^m)(1-e^{-\frac{2\pi in}{N}}q^m)(1-q^m)^{-2}\bigg)^{-1}\\
&=&e^{\frac{\pi i(N-1)}{2}(\frac{t}{N}-1)}q^{\frac{t}{2}(\frac{t}{N}-1)}(1-q^t)\prod_{m=1}^\infty(1-q^{Nm+t})(1-q^{Nm-t})(1-q^m)^{-2N}\\
&&\times \bigg(e^{-\frac{\pi i(N-1)}{2}}N
\prod_{m=1}^\infty(1-q^{Nm})(1-q^m)^{-1}(1-q^{Nm})(1-q^m)^{-1}(1-q^m)^{-2(N-1)}\bigg)^{-1}\\
&=&e^{\frac{\pi it}{2}(1-\frac{1}{N})}N^{-1}
q^{\frac{t}{2}(\frac{t}{N}-1)} (1-q^t)
\prod_{m=1}^\infty(1-q^{Nm+t})(1-q^{Nm-t})(1-q^{Nm})^{-2}
\end{eqnarray*}
by using the identity
\begin{equation}\label{trivialidentity}
1-X^N=(1-\zeta_NX)(1-\zeta_N^2X)\cdots(1-\zeta_N^NX)\quad\textrm{where}~\zeta_N=e^{\frac{2\pi
i}{N}}.
\end{equation}
Hence we get the assertion.
\end{proof}

\begin{lemma}\label{Bernoulli}
For $y\in\mathbb{Q}$ and an integer $D\geq1$ we have
$$\sum_{\begin{smallmatrix}x\pmod{\mathbb{Z}}\\Dx\equiv y\pmod{\mathbb{Z}}\end{smallmatrix}}\mathbf{B}_2\big(\langle
x\rangle\big)=D^{-1}\mathbf{B}_2\big(\langle y\rangle\big).$$
\end{lemma}
\begin{proof}
See \cite{K-S} Lemma 6.3.
\end{proof}

\begin{theorem}\label{modularity2}
For an integer $N\geq2$, let $\{m(t)\}_{t=1}^{N-1}$ be a family of
integers. Then the product
\begin{equation*}
\prod_{t=1}^{N-1}\mathfrak{k}_{(\frac{t}{N},0)}(N\tau)^{m(t)}
\end{equation*}
is a nearly holomorphic modular form for $\Gamma_1(N)$ of weight
$k=-\sum_{t=1}^{N-1} m(t)$ if
\begin{equation}\label{condition}
\sum_{t=1}^{N-1} m(t)t^2\equiv0\pmod{\gcd(2,N)\cdot N}.
\end{equation}
Furthermore, for
$\alpha=\begin{pmatrix}a&b\\c&d\end{pmatrix}\in\mathrm{SL}_2(\mathbb{Z})$
we achieve
\begin{equation}\label{order}
\mathrm{ord}_{q}\bigg(\prod_{t=1}^{N-1}\mathfrak{k}_{(\frac{t}{N},0)}(N\tau)^{m(t)}
|[\alpha]_k\bigg)=\frac{\gcd(c,N)^2}{2N}\sum_{t=1}^{N-1} m(t)
\bigg\langle\frac{at}{\gcd(c,N)}\bigg\rangle\bigg(\bigg\langle\frac{at}{\gcd(c,N)}\bigg\rangle-1\bigg).
\end{equation}
\end{theorem}
\begin{proof}
By Lemma \ref{distribution} we may prove the assertions for the
function
\begin{equation*}
\mathfrak{k}(\tau)=
\prod_{t=1}^{N-1}\bigg(\prod_{n=0}^{N-1}\mathfrak{k}_{(\frac{t}{N},\frac{n}{N})}(\tau)
\prod_{n=1}^{N-1}\mathfrak{k}_{(0,\frac{n}{N})}(\tau)^{-1}\bigg)^{m(t)}.
\end{equation*}
Assume the condition (\ref{condition}) and set
\begin{equation*}
\mathfrak{k}(\tau)=\prod_{r=(r_1,r_2)\in\frac{1}{N}\mathbb{Z}^2-\mathbb{Z}^2}
\mathfrak{k}_r(\tau)^{\ell(r)}.
\end{equation*}
Then we get that
\begin{eqnarray*}
&&\sum_r \ell(r)(Nr_1)^2=N\sum_{t=1}^{N-1}m(t)t^2\equiv0\pmod{\gcd(2,N)\cdot N}\\
&&\sum_r \ell(r)(Nr_2)^2=0\equiv0\pmod{\gcd(2,N)\cdot N}\\
&&\sum_r
\ell(r)(Nr_1)(Nr_2)=\frac{N(N-1)}{2}\sum_{t=1}^{N-1}m(t)t\equiv0\pmod{N}
\end{eqnarray*}
by the condition (\ref{condition}) and the fact $\sum_t
m(t)t\equiv\sum_t m(t)t^2\pmod{2}$. This shows that
$\mathfrak{k}(\tau)$ is a nearly holomorphic modular form for
$\Gamma(N)$ of weight $k=-\sum_{t=1}^{N-1} m(t)$ by Proposition
\ref{modularity}.
\par On the other hand, we know that
$\Gamma_1(N)$ is generated by $\Gamma(N)$ and
$T=\begin{pmatrix}1&1\\0&1\end{pmatrix}$. Thus, we derive that
\begin{eqnarray*}
&&\mathfrak{k}(\tau)|[T]_k\\
&=&\prod_{t=1}^{N-1}\bigg(\prod_{n=0}^{N-1}\mathfrak{k}_{(\frac{t}{N},\frac{t+n}{N})}(\tau)
\prod_{n=1}^{N-1}\mathfrak{k}_{(0,\frac{n}{N})}(\tau)^{-1}\bigg)^{m(t)}\quad
\textrm{by Proposition \ref{TransformKlein}(ii)}\\
&=&\prod_{t=1}^{N-1}\bigg(\prod_{n=0}^{N-1-t}\mathfrak{k}_{(\frac{t}{N},\frac{t+n}{N})}(\tau)
\prod_{n=N-t}^{N-1}\mathfrak{k}_{(\frac{t}{N},\frac{t+n-N}{N})+(0,1)}(\tau)
\prod_{n=1}^{N-1}\mathfrak{k}_{(0,\frac{n}{N})}(\tau)^{-1}\bigg)^{m(t)}\\
&=&\prod_{t=1}^{N-1}\bigg(\prod_{n=0}^{N-1-t}\mathfrak{k}_{(\frac{t}{N},\frac{t+n}{N})}(\tau)
\prod_{n=N-t}^{N-1}(-e^{-\pi
i\frac{t}{N}})\mathfrak{k}_{(\frac{t}{N},\frac{t+n-N}{N})}(\tau)
\prod_{n=1}^{N-1}\mathfrak{k}_{(0,\frac{n}{N})}(\tau)^{-1}\bigg)^{m(t)}\quad\textrm{by
Proposition \ref{TransformKlein}(i)}\\
&=&(-1)^{\sum_t m(t)t}e^{-\pi i\frac{1}{N}\sum_t m(t)t^2}\mathfrak{k}(\tau)\\
&=&\mathfrak{k}(\tau)\quad\textrm{by the condition (\ref{condition})
and the fact $\textstyle\sum_t m(t)t\equiv\sum_t m(t)t^2\pmod{2}$}.
\end{eqnarray*}
Therefore $\mathfrak{k}(\tau)$ is a nearly holomorphic modular form
for $\Gamma_1(N)$ of weight $k=-\sum_t m(t)$.
\par
Now, for
$\alpha=\begin{pmatrix}a&b\\c&d\end{pmatrix}\in\mathrm{SL}_2(\mathbb{Z})$
we deduce that
\begin{eqnarray*}
&&\mathrm{ord}_{q}(\mathfrak{k}(\tau)|[\alpha]_k)\\
&=&\mathrm{ord}_{q}\bigg(
\prod_{t=1}^{N-1}\bigg(\prod_{n=0}^{N-1}\mathfrak{k}_{(\frac{at+cn}{N},\frac{bt+dn}{N})}(\tau)
\prod_{n=1}^{N-1}\mathfrak{k}_{(\frac{cn}{N},\frac{dn}{N})}(\tau)^{-1}\bigg)^{m(t)}\bigg)\quad
\textrm{by Proposition \ref{TransformKlein}(ii)}\\
&=&\sum_{t=1}^{N-1}m(t)\bigg\{\sum_{n=0}^{N-1}\frac{1}{2}\bigg(\mathbf{B}_2\bigg(\bigg\langle\frac{at+cn}{N}\bigg\rangle
\bigg)-\frac{1}{6}\bigg)-\sum_{n=1}^{N-1}\frac{1}{2}\bigg(\mathbf{B}_2\bigg(\bigg\langle\frac{cn}{N}\bigg\rangle
\bigg)-\frac{1}{6}\bigg)\bigg\}\quad\textrm{by Proposition \ref{TransformKlein}(iii)}\\
&=&\frac{1}{2}\sum_{t=1}^{N-1}m(t)\bigg\{\sum_{n=1}^{N}\mathbf{B}_2\bigg(\bigg\langle\frac{at+cn}{N}\bigg\rangle
\bigg)-\sum_{n=1}^{N}\mathbf{B}_2\bigg(\bigg\langle\frac{cn}{N}\bigg\rangle
\bigg)\bigg\}\quad\textrm{by the fact
$\mathbf{B}_2(0)=\frac{1}{6}$}\\
&=&\frac{\gcd(c,N)}{2}\sum_{t=1}^{N-1}m(t)\bigg\{\sum_{n=1}^D\mathbf{B}_2\bigg(\bigg\langle\frac{at/\gcd(c,N)}{D}+\frac{c/\gcd(c,N)}{D}n\bigg\rangle\bigg)-
\sum_{n=1}^D\mathbf{B}_2\bigg(\bigg\langle\frac{c/\gcd(c,N)}{D}n\bigg\rangle\bigg)\bigg\}\\
&&\textrm{where}~D=\frac{N}{\gcd(c,N)}.
\end{eqnarray*}
If we apply Lemma \ref{Bernoulli} with $D=\frac{N}{\gcd(c,N)}$,
$y=\frac{at}{\gcd(c,N)}$ and $x=\frac{y}{D}+\frac{c/\gcd(c,N)}{D}n$
with $1\leq n\leq D$, then we obtain
\begin{equation*}
\sum_{n=1}^D\mathbf{B}_2\bigg(\bigg\langle\frac{at/\gcd(c,N)}{D}+\frac{c/\gcd(c,N)}{D}n\bigg\rangle\bigg)
=\frac{\gcd(c,N)}{N}\mathbf{B}_2\bigg(\bigg\langle\frac{at}{\gcd(c,N)}\bigg\rangle\bigg).
\end{equation*}
Likewise, if we set $D=\frac{N}{\gcd(c,N)}$, $y=0$ and
$x=\frac{c/\gcd(c,N)}{D}$ with $1\leq n\leq D$ in Lemma
\ref{Bernoulli}, then we get
\begin{equation}\label{cf}
\sum_{n=1}^D\mathbf{B}_2\bigg(\bigg\langle\frac{c/\gcd(c,N)}{D}n\bigg\rangle\bigg)
=\frac{\gcd(c,N)}{N}\mathbf{B}_2(0).
\end{equation}
Hence we achieve
\begin{eqnarray*}
\mathrm{ord}_{q}(\mathfrak{k}(\tau)|[\alpha]_k) &=&
\frac{\gcd(c,N)^2}{2N}\sum_{t=1}^{N-1}m(t)\bigg\{
\mathbf{B}_2\bigg(\bigg\langle\frac{at}{\gcd(c,N)}\bigg\rangle\bigg)-\mathbf{B}_2(0)\bigg\}
\\
&=&\frac{\gcd(c,N)^2}{2N}\sum_{t=1}^{N-1}m(t)
\bigg\langle\frac{at}{\gcd(c,N)}\bigg\rangle\bigg(\bigg\langle\frac{at}{\gcd(c,N)}\bigg\rangle-1\bigg),
\end{eqnarray*}
as desired.
\end{proof}

\begin{corollary}\label{squarelevel}
Let $N\geq2$ be a square integer. Then the function
\begin{equation*}
\mathfrak{k}_{(\frac{\sqrt{N}}{N},0)}(N\tau)^{-2}
\end{equation*}
belongs to $M_2(\Gamma_1(N))$.
\end{corollary}
\begin{proof}
Let $\mathfrak{k}(\tau)$ be the above function. Since
$\mathfrak{k}(\tau)$ satisfies the condition (\ref{condition}), it
is a nearly holomorphic modular form for $\Gamma_1(N)$ of weight $2$
by Theorem \ref{modularity2}. For any
$\alpha=\begin{pmatrix}a&b\\c&d\end{pmatrix}\in\mathrm{SL}_2(\mathbb{Z})$
we then get by the order formula (\ref{order})
\begin{equation*}
\mathrm{ord}_{q}(\mathfrak{k}(\tau)|[\alpha]_2)=\frac{\gcd(c,N)^2}{N}
\bigg\langle\frac{a\sqrt{N}}{\gcd(c,N)}\bigg\rangle\bigg(1-\bigg\langle\frac{a\sqrt{N}}{\gcd(c,N)}\bigg\rangle\bigg),
\end{equation*}
which is nonnegative. This implies that the order of
$\mathfrak{k}(\tau)$ at every cusp is nonnegative; hence
$\mathfrak{k}(\tau)$ is indeed a modular form.
\end{proof}

Next we find a family of modular forms for $\Gamma_0(N)$ which are
in fact quotients of the Dedekind eta-functions.

\begin{theorem}\label{etaquotient}
For an integer $N\geq2$ the function
\begin{equation*}
\prod_{n=1}^{N-1}\mathfrak{k}_{(0,
\frac{n}{N})}(\tau)^{\frac{-12}{\gcd(12, N-1)}}
\end{equation*}
is a modular form for $\Gamma_0(N)$ of weight
$\frac{12(N-1)}{\gcd(12,N-1)}$.
\end{theorem}
\begin{proof}
Let $\mathfrak{k}(\tau)$ be the above function,
$k=\frac{12(N-1)}{\gcd(12,N-1)}$ and
$\alpha=\begin{pmatrix}a&b\\Nc&d\end{pmatrix}$ with
$a,b,c,d\in\mathbb{Z}$ such that $ad-Ncb=1$. Then we achieve
\begin{eqnarray*}
\mathfrak{k}(\tau)|[\alpha]_k &=&\prod_{n=1}^{N-1}\mathfrak{k}_{(cn,
\frac{dn}{N})}(\tau)^{\frac{-12}{\gcd(12, N-1)}}
\quad\textrm{by Proposition \ref{TransformKlein}(ii)}\\
&=&\prod_{n=1}^{N-1}\mathfrak{k}_{(0,
\frac{dn}{N})+(cn,0)}(\tau)^{\frac{-12}{\gcd(12, N-1)}}\\
&=&\prod_{n=1}^{N-1}\bigg(\mathfrak{k}_{(0,
\frac{dn}{N})}(\tau)(-1)^{cn}e^{-\pi
i\frac{cdn^2}{N}}\bigg)^{\frac{-12}{\gcd(12, N-1)}}\quad\textrm{by
Proposition \ref{TransformKlein}(i)}\\
&=&\bigg(\prod_{n=1}^{N-1}\mathfrak{k}_{(0,
\frac{dn}{N})}(\tau)^{\frac{-12}{\gcd(12,
N-1)}}\bigg)\bigg((-1)^{\frac{c(N-1)N}{2}}e^{-\pi
i\frac{cd(N-1)(2N-1)}{6}}\bigg)^{\frac{-12}{\gcd(12, N-1)}}\\
&=&\prod_{n=1}^{N-1}\mathfrak{k}_{(0,
\langle\frac{dn}{N}\rangle)+(0,\frac{dn}{N}-\langle\frac{dn}{N}\rangle)}(\tau)^{\frac{-12}{\gcd(12,
N-1)}}\\
&=&\prod_{n=1}^{N-1}\bigg(\mathfrak{k}_{(0,
\langle\frac{dn}{N}\rangle)}(\tau)(-1)^{\frac{dn}{N}-\langle\frac{dn}{N}\rangle}\bigg)^{\frac{-12}{\gcd(12,
N-1)}}\quad\textrm{by
Proposition \ref{TransformKlein}(i)}\\
&=&\bigg(\prod_{n=1}^{N-1}\mathfrak{k}_{(0,
\langle\frac{dn}{N}\rangle)}(\tau)^{\frac{-12}{\gcd(12, N-1)}}\bigg)
\bigg((-1)^{\sum_n\frac{dn}{N}-\sum_n\langle\frac{dn}{N}\rangle}\bigg)^{\frac{-12}{\gcd(12,
N-1)}}\\
 &=&\bigg(\prod_{n=1}^{N-1}\mathfrak{k}_{(0,
\langle\frac{n}{N}\rangle)}(\tau)^{\frac{-12}{\gcd(12, N-1)}}\bigg)
\bigg((-1)^{\sum_n\frac{dn}{N}-\sum_n\langle\frac{n}{N}\rangle}\bigg)^{\frac{-12}{\gcd(12,
N-1)}}\\
&=&\bigg(\prod_{n=1}^{N-1}\mathfrak{k}_{(0,
\frac{n}{N})}(\tau)\bigg)^{\frac{-12}{\gcd(12, N-1)}}
\bigg((-1)^{\frac{(d-1)(N-1)}{2}}\bigg)^{\frac{-12}{\gcd(12,
N-1)}}\\
&=&\mathfrak{k}(\tau).
\end{eqnarray*}
Hence $\mathfrak{k}(\tau)$ is a nearly holomorphic modular form for
$\Gamma_0(N)$ of weight $k=\frac{12(N-1)}{\gcd(12,N-1)}$.
\par
Now let
$\beta=\begin{pmatrix}x&y\\z&w\end{pmatrix}\in\mathrm{SL}_2(\mathbb{Z})$.
Then we obtain that
\begin{eqnarray*}
\mathrm{ord}_{q}(\mathfrak{k}(\tau)|[\beta]_k)&=&
\frac{-12}{\gcd(12,N-1)}\mathrm{ord}_{q}\bigg(\prod_{n=1}^{N-1}
\mathfrak{k}_{(\frac{nz}{N},\frac{nw}{N})}(\tau)\bigg)\quad\textrm{by
Proposition \ref{TransformKlein}(ii)}\\
&=&\frac{-12}{\gcd(12,N-1)}\sum_{n=1}^{N-1}\frac{1}{2}\bigg(\mathbf{B}_2\bigg(\bigg\langle\frac{nz}{N}\bigg\rangle
\bigg)-\frac{1}{6}\bigg)\quad\textrm{by Proposition
\ref{TransformKlein}(iii)}\\
&=&\frac{-6}{\gcd(12,N-1)}
\bigg(\sum_{n=1}^{N}\mathbf{B}_2\bigg(\bigg\langle\frac{nz}{N}\bigg\rangle
\bigg)-\frac{N}{6}\bigg)\quad\textrm{by the fact $\mathbf{B}_2(0)=\frac{1}{6}$}\\
&=&\frac{-6}{\gcd(12,N-1)}
\bigg(\gcd(z,N)\sum_{n=1}^{\frac{N}{\gcd(z,N)}}\mathbf{B}_2\bigg(\bigg\langle\frac{nz}{N}\bigg\rangle
\bigg)-\frac{N}{6}\bigg)\\
&=&\frac{-6}{\gcd(12,N-1)}
\bigg(\frac{\gcd(z,N)^2}{N}\mathbf{B}_2(0)-\frac{N}{6}\bigg)\quad\textrm{by
the same argument as (\ref{cf})}\\
&=&\frac{N^2-\gcd(z,N)^2}{\gcd(12,N-1)\cdot N}\geq0,
\end{eqnarray*}
which yields that $\mathfrak{k}(\tau)$ is holomorphic at every cusp.
This completes the proof.
\end{proof}

\begin{remark}
Using the identity (\ref{trivialidentity}) one is readily able to
verify that the function in Theorem \ref{etaquotient} can be written
as
\begin{equation*}
\mathfrak{k}(\tau)=\bigg(N\frac{\eta(N\tau)^2}{\eta(\tau)^{2N}}
\bigg)^{\frac{-12}{\gcd(12, N-1)}}
\end{equation*}
by the definitions (\ref{KleinForm}) and (\ref{eta}). So, we may
regard the following general theorem about the Dedekind eta-function
as the first part of the above proof.
\end{remark}

\begin{theorem}
Let $N$ be a positive integer. If
$f(\tau)=\prod_{\delta|N}\eta(\delta\tau)^{r_\delta}$ is an
eta-quotient with
$k=\frac{1}{2}\sum_{\delta|N}r_\delta\in\mathbb{Z}$, with the
additional properties that
\begin{equation*}
\sum_{\delta|N}\delta r_\delta\equiv0\pmod{24}
\quad\textrm{and}\quad
\sum_{\delta|N}\frac{N}{\delta}r_\delta\equiv0\pmod{24},
\end{equation*}
then $f(\tau)$ satisfies
\begin{equation*}
f\bigg(\frac{a\tau+b}{c\tau+d}\bigg)=\chi(d)(c\tau+d)^k f(\tau)
\end{equation*}
for every $\begin{pmatrix}a&b\\c&d\end{pmatrix}\in\Gamma_0(N)$. Here
the character $\chi$ is defined by
\begin{equation*}
\chi(d)=\textrm{the Kronecker
symbol}~\bigg(\frac{(-1)^k\prod_{\delta|N}\delta^{r_\delta}}{d}\bigg).
\end{equation*}
\end{theorem}
\begin{proof}
See \cite{Ono} Theorem 1.64.
\end{proof}

\section{Theta functions}

Let $N\geq1$ and $k$ be integers. For a Dirichlet character $\chi$
modulo $N$ we define a character of $\Gamma_0(N)$, also denoted by
$\chi$, by
\begin{equation*}
\chi(\gamma)=\chi(d)~\textrm{for}~\gamma=\begin{pmatrix}a&b\\c&d\end{pmatrix}\in\Gamma_0(N).
\end{equation*}
If we let $M_k(\Gamma_0(N),\chi)$ be the space
\begin{equation*}
\big\{f(\tau)\in M_k(\Gamma_1(N))~:~
f(\tau)|[\gamma]_k=\chi(\gamma)f(\tau)~\textrm{for
all}~\gamma\in\Gamma_0(N)\big\},
\end{equation*}
then we have the following decomposition.

\begin{proposition}\label{decomposition}
Let $N\geq1$ and $k$ be integers. We have
\begin{equation*}
M_k(\Gamma_1(N))=\bigoplus_\chi M_k(\Gamma_0(N),\chi)
\end{equation*}
where $\chi$ runs over all Dirichlet characters modulo $N$. If
$\chi(-1)\neq(-1)^k$, then $M_k(\Gamma_0(N),\chi)=\{0\}$.
\end{proposition}
\begin{proof}
See \cite{Miyake} Lemmas 4.3.1 and 4.3.2.
\end{proof}

Let $A$ be an $r\times r$ positive definite symmetric matrix over
$\mathbb{Z}$ with even diagonal entries and $Q$ be its associated
quadratic form, namely
\begin{equation*}
Q=Q(\mathbf{x})=\frac{1}{2}\mathbf{x}A\mathbf{x}^t ~\textrm{for}~
\mathbf{x}=(x_1,\cdots,x_r)\in\mathbb{Z}^r.
\end{equation*}
Now, define the theta function $\Theta_Q(\tau)$ on $\mathfrak{H}$
associated with $Q$ by
\begin{equation*}
\Theta_Q(\tau)=\sum_{\mathbf{x}\in\mathbb{Z}^r}e^{2\pi
iQ(\mathbf{x})\tau} =\sum_{n=0}^\infty r_Q(n)q^n
\end{equation*}
where
\begin{equation*}
r_Q(n)=\#\{\mathbf{x}\in\mathbb{Z}^r:Q(\mathbf{x})=n\}.
\end{equation*}
We take a positive integer $N$ such that $NA^{-1}$ is an integral
matrix with even diagonal entries.

\begin{proposition}\label{theta}
With the notations as above we further assume that $r$ is even. Then
$\Theta_Q(\tau)$ is a modular form for $\Gamma_1(N)$ of weight
$r/2$. More precisely, $\Theta_Q(\tau)$ belongs to
$M_{r/2}(\Gamma_0(N),\chi)$ where $\chi$ is a Dirichlet character
defined by
\begin{equation*}
\chi(d)=\textrm{the Kronecker
symbol}~\bigg(\frac{(-1)^\frac{r}{2}\det(A)}{d}
\bigg)~\textrm{for}~d\in\mathbb{Z}-N\mathbb{Z}.
\end{equation*}
\end{proposition}
\begin{proof}
See \cite{Miyake} Corollary 4.9.5.
\end{proof}

\begin{example}\label{twovariables}
Let $A=\begin{pmatrix} 2&0\\0&2\end{pmatrix}$. Its associated
quadratic form is $Q=x_1^2+x_2^2$ and
\begin{eqnarray*}
\Theta_Q(\tau)&=&\sum_{n=0}^\infty
\#\{(x_1,x_2)\in\mathbb{Z}^2:x_1^2+x_2^2=n\}q^n =1+4q+4q^2+\cdots
\\
&=&\bigg(\sum_{x_1=-\infty}^\infty q^{x_1^2}\bigg)
\bigg(\sum_{x_2=-\infty}^\infty q^{x_2^2}\bigg) =\Theta(\tau)^2.
\end{eqnarray*}
It follows from Proposition \ref{theta} that
$\Theta_Q(\tau)=\Theta(\tau)^2$ belongs to $M_1(\Gamma_1(4))$. On
the other hand, since $M_1(\Gamma_1(4))$ is of dimension $1$
(\cite{Shimura} $\S$2.6) and the function
\begin{equation*}
\mathfrak{k}_{(\frac{1}{4},0)}(4\tau)^{-4}
\mathfrak{k}_{(\frac{2}{4},0)}(4\tau)^3=1+4q+4q^2+\cdots
\end{equation*}
is in $M_1(\Gamma_1(4))$ by Theorem \ref{modularity2}, we obtain
$\Theta(\tau)^2=\mathfrak{k}_{(\frac{1}{4},0)}(4\tau)^{-4}
\mathfrak{k}_{(\frac{2}{4},0)}(4\tau)^3$.
\par
Furthermore, we derive from the definition (\ref{KleinForm}) that
\begin{eqnarray*}
\mathfrak{k}_{(\frac{1}{4},0)}(4\tau)^{-4}
\mathfrak{k}_{(\frac{2}{4},0)}(4\tau)^3&=& \bigg(
q^{\frac{1}{2}(\frac{1}{4}-1)}(1-q) \prod_{n=1}^\infty
(1-q^{4n+1})(1-q^{4n-1})(1-q^{4n})^{-2}
 \bigg)^{-4}\\
 &&\times\bigg(
q^{\frac{2}{2}(\frac{2}{4}-1)}(1-q^2) \prod_{n=1}^\infty
(1-q^{4n+2})(1-q^{4n-2})(1-q^{4n})^{-2}\bigg)^3\\
&=&\prod_{n=1}^\infty \bigg((1-q^{4n-3})^{-4}(1-q^{4n-1})^{-4}\bigg)
\bigg(
(1-q^{4n-2})^2(1-q^{4n})^2\bigg)(1-q^{4n-2})^4\\
&=&\prod_{n=1}^\infty
(1-q^{2n-1})^{-4}(1-q^{2n})^2(1-q^{2n-1})^4(1+q^{2n-1})^4\\
&=&\prod_{n=1}^\infty(1-q^{2n})^2(1+q^{2n-1})^4\\
&=&\prod_{n=1}^\infty\bigg((1-(-q)^{2n})^4(1-(-q)^{2n-1})^4\bigg)(1-(-q)^{2n})^{-2}\\
&=&\prod_{n=1}^\infty(1-(-q)^n)^4(1-(-q)^n)^{-2}(1+(-q)^n)^{-2}\\
&=&\prod_{n=1}^\infty\bigg( \frac{1-(-q)^n}{1+(-q)^n} \bigg)^{2}.
\end{eqnarray*}
Therefore, we get an infinite product formula for $\Theta(\tau)^2$.
This gives a simple example of utilizing Klein forms, however, one
can obtain this result from (\ref{firsttheta}) more easily.
\end{example}

\begin{example}\label{fourvariables}
If $A=\begin{pmatrix}
2&0&1&1\\0&4&0&1\\1&0&2&0\\1&1&0&2\end{pmatrix}$, then $A$ has
positive eigenvalues $\frac{5}{2}\pm\frac{\sqrt{9\pm4\sqrt{3}}}{2}$,
which shows that $A$ is positive definite. Its associated quadratic
form is $Q=x_1^2+2x_2^2+x_3^2+x_4^2+x_1x_3+x_1x_4+x_2x_4$ and the
theta function $\Theta_Q(\tau)$ has the expansion
\begin{equation}\label{r_Q}
\Theta_Q(\tau)=
1+12q+14q^2+48q^3+36q^4+56q^5+56q^6+84q^7+70q^8+156q^9+48q^{10}
+140q^{11}+144q^{12}+\cdots.
\end{equation}
If $N=13$, then $NA^{-1}=
\begin{pmatrix}14&2&-7&-8\\2&4&-1&-3\\
-7&-1&10&4\\-8&-3&4&12\end{pmatrix}$ has even diagonal entries.
Hence $\Theta_Q(\tau)$ belongs to $M_2(\Gamma_1(13))$ by Proposition
\ref{theta}.
\par
We know that $M_2(\Gamma_1(13))$ is of dimension $13$
(\cite{Shimura} $\S$2.6) and all the inequivalent cusps for
$\Gamma_1(13)$ are given by
\begin{equation}\label{cuspexample}
\frac{a}{c}=\frac{1}{1},~\frac{1}{2},~\frac{1}{3},~\frac{1}{4},~\frac{1}{5},~\frac{1}{6},~
\frac{1}{13},~\frac{2}{13},~\frac{3}{13},~\frac{4}{13},~\frac{5}{13},~\frac{6}{13}
\end{equation}
(\cite{Shimura} $\S$1.6). Consider a function
\begin{equation*}
\mathfrak{k}(\tau)=\prod_{t=1}^{6}\mathfrak{k}_{(\frac{t}{13},0)}(13\tau)^{m(t)}\quad
\textrm{with}~m(1),\cdots,m(6)\in\mathbb{Z}.
\end{equation*}
(By Proposition \ref{TransformKlein}(i) we confine ourselves to the
case $1\leq t\leq6$.) For each cusp $a/c$ we take a matrix
$\alpha_{a/c}\in\mathrm{SL}_2(\mathbb{Z})$ so that
$\alpha_{a/c}(\infty)=a/c$, for example
\begin{equation*}
\begin{array}{llll}
\alpha_{1/1}=\begin{pmatrix} 1&0\\1&1
\end{pmatrix}, &
\alpha_{1/2}=\begin{pmatrix} 1&0\\2&1
\end{pmatrix}, &
\alpha_{1/3}=\begin{pmatrix} 1&0\\3&1
\end{pmatrix}, &
\alpha_{1/4}=\begin{pmatrix} 1&0\\4&1
\end{pmatrix},
\vspace{0.2cm}\\
\alpha_{1/5}=\begin{pmatrix} 1&0\\5&1
\end{pmatrix}, &
\alpha_{1/6}=\begin{pmatrix} 1&0\\6&1
\end{pmatrix}, &
\alpha_{1/13}=\begin{pmatrix} 1&0\\13&1
\end{pmatrix}, &
\alpha_{2/13}=\begin{pmatrix} 2&1\\13&7
\end{pmatrix},
\vspace{0.2cm}
\\
\alpha_{3/13}=\begin{pmatrix} 3&-1\\13&-4
\end{pmatrix}, &
\alpha_{4/13}=\begin{pmatrix} 4&-1\\13&-3
\end{pmatrix}, &
\alpha_{5/13}=\begin{pmatrix} 5&-2\\13&-5
\end{pmatrix}, &
\alpha_{6/13}=\begin{pmatrix} 6&-1\\13&-2
\end{pmatrix}.
\end{array}
\end{equation*}
Note that $\alpha_{a/c}=\begin{pmatrix}a&*\\c&*\end{pmatrix}$. We
then obtain a criterion by Theorem \ref{modularity2} for
$\mathfrak{k}(\tau)$ to belong to the space $M_2(\Gamma_1(13))$,
namely
\begin{eqnarray*}
&&\sum_{t=1}^6 m(t)=-2,\quad
\sum_{t=1}^6 m(t)t^2\equiv0\pmod{13}\quad\textrm{and}\\
&&\mathrm{ord}_q\bigg(\mathfrak{k}(\tau)|[\alpha_{a/c}]_2\bigg)
=\frac{\gcd(c,13)^2}{26}\sum_{t=1}^6 m(t)
\bigg\langle\frac{at}{\gcd(c,13)}\bigg\rangle\bigg(\bigg\langle\frac{at}{\gcd(c,13)}\bigg\rangle-1\bigg)\geq0
\quad\textrm{for all}~\frac{a}{c}~\textrm{in
(\ref{cuspexample})}.\nonumber
\end{eqnarray*}
Thus one can readily find such $\mathfrak{k}(\tau)$'s as in the
following Table \ref{Table}. Here we use the notation
\begin{equation*}
\prod_{t=1}^6
(t)^{m(t)}:=\prod_{t=1}^{6}\mathfrak{k}_{(\frac{t}{13},0)}(13\tau)^{m(t)}.
\end{equation*}

\begin{table}[h]
\begin{tabular}
{|c|c|}\hline $\mathfrak{k}(\tau)$ & $\mathrm{ord}_{q}(\mathfrak{k}(\tau)|[\alpha_{6/13}]_2)$\\
\hline \hline $\footnotesize\begin{array}{l}\mathfrak{k}_1(\tau):=
(1)^{-3}(2)^{-2}(3)^{5}(4)^{-2}(5)^{-1}(6)^{1}
\\=1+3q+8q^2+11q^3+17q^4+17q^5+28q^6+26q^7+39q^8+27q^9+48q^{10}+35q^{11}+59q^{12}+\cdots\end{array}$
& $0$\\
\hline $\footnotesize\begin{array}{l}\mathfrak{k}_2(\tau):=
(1)^{-4}(2)^{1}(3)^{3}(4)^{-5}(5)^{5}(6)^{-2}
\\=1+4q+9q^2+13q^3+18q^4+24q^5+31q^6+31q^7+36q^8+44q^9+54q^{10}+46q^{11}+47q^{12}+\cdots\end{array}$
& $1$\\
\hline
$\footnotesize\begin{array}{l}\mathfrak{k}_3(\tau):=(1)^{-4}(3)^{4}(4)^{-2}\\
=1+4q+10q^2+16q^3+21q^4+24q^5+30q^6+36q^7+42q^8+46q^9+54q^{10}+60q^{11}+59q^{12}+\cdots\end{array}$
& $2$\\
\hline $\footnotesize\begin{array}{l}\mathfrak{k}_4(\tau):=
(1)^{-4}(2)^{-1}(3)^{5}(4)^{1}(5)^{-5}(6)^{2}\\
=1+4q+11q^2+19q^3+25q^4+26q^5+27q^6+36q^7+49q^8+59q^9+59q^{10}+57q^{11}+66q^{12}+\cdots\end{array}$
& $3$\\
\hline $\footnotesize\begin{array}{l}\mathfrak{k}_5(\tau):=
(1)^{-4}(3)^{3}(4)^{1}(5)^{-3}(6)^{1}\\
=1+4q+10q^2+17q^3+22q^4+25q^5
+28q^6+35q^7+44q^8+51q^9+56q^{10}+57q^{11}+59q^{12}+\cdots\end{array}$
& $4$\\
\hline $\footnotesize\begin{array}{l}\mathfrak{k}_6(\tau):=
(1)^{-4}(2)^{1}(3)^{1}(4)^{1}(5)^{-1}\\
=1+4q+9q^2+15q^3+20q^4+24q^5+28q^6+33q^7+40q^8+47q^9
+52q^{10}+53q^{11}+53q^{12}+\cdots\end{array}$
& $5$\\
\hline $\footnotesize\begin{array}{l}\mathfrak{k}_7(\tau):=
(1)^{-4}(2)^{1}(3)^{1}(4)^{2}(5)^{-4}(6)^{2}\\
=1+4q+9q^2+15q^3+19q^4+23q^5+29q^6+35q^7+42q^8+43q^9+45q^{10}+53q^{11}+56q^{12}+\cdots\end{array}$
& $6$\\
\hline $\footnotesize\begin{array}{l}\mathfrak{k}_8(\tau):=
(1)^{-5}(2)^{2}(3)^{2}(4)^{2}(5)^{-5}(6)^{2}\\
=1+5q+13q^2+23q^3+29q^4+30q^5+33q^6+43q^7+59q^8+67q^9+66q^{10}+71q^{11}+79q^{12}+\cdots\end{array}$
& $7$\\
\hline $\footnotesize\begin{array}{l}\mathfrak{k}_9(\tau):=
(1)^{-5}(2)^{3}(4)^{2}(5)^{-3}(6)^{1}\\
=1+5q+12q^2+20q^3+26q^4+29q^5+34q^6+42q^7+51q^8+60q^9+64q^{10}+68q^{11}
+72q^{12}+\cdots\end{array}$
& $8$\\
\hline $\footnotesize\begin{array}{l}\mathfrak{k}_{10}(\tau):=
(1)^{-5}(2)^{4}(3)^{-2}(4)^{2}(5)^{-1}\\
=1+5q+11q^2+17q^3+24q^4+29q^5+32q^6
+40q^7+48q^8+53q^9+61q^{10}+64q^{11}+62q^{12}+\cdots\end{array}$
& $9$\\
\hline $\footnotesize\begin{array}{l}\mathfrak{k}_{11}(\tau):=
(1)^{-5}(2)^{5}(3)^{-3}(4)^{1}(5)^{-2}(6)^{2}\\
=1+5q+10q^2+13q^3+19q^4+28q^5
+34q^6+40q^7+41q^8+40q^9+53q^{10}+60q^{11}+54q^{12}+\cdots\end{array}$
& $10$\\
\hline $\footnotesize\begin{array}{l}\mathfrak{k}_{12}(\tau):=
(1)^{-5}(2)^{5}(3)^{-4}(4)^{3}(5)^{-2}(6)^{1}\\
=1+5q+10q^2+14q^3+22q^4+28q^5
+29q^6+42q^7+47q^8+39q^9+58q^{10}+60q^{11}+47q^{12}+\cdots\end{array}$
& $11$\\
\hline $\footnotesize\begin{array}{l}\mathfrak{k}_{13}(\tau):=
(1)^{-6}(2)^{6}(3)^{-3}(4)^{3}(5)^{-3}(6)^{1}\\
=1+6q+15q^2+23q^3+30q^4+36q^5
+39q^6+50q^7+63q^8+65q^9+76q^{10}+84q^{11}+81q^{12}+\cdots\end{array}$
& $12$\\
\hline
\end{tabular}
\caption{Modular forms for $\Gamma_1(13)$ of weight
$2$}\label{Table}
\end{table}
Since $\mathrm{ord}_{q}(\mathfrak{k}_m(\tau)|[\alpha_{6/13}]_2)$
($m=1,\cdots,13$) are all distinct, the set
$\{\mathfrak{k}_1(\tau),\cdots,\mathfrak{k}_{13}(\tau)\}$ forms a
basis of $M_2(\Gamma_1(13))$ over $\mathbb{C}$. Hence
$\Theta_Q(\tau)$ is a linear combination of these
$\mathfrak{k}_m(\tau)$ over $\mathbb{C}$, namely
\begin{equation}\label{equation}
\Theta_Q(\tau)=\sum_{n=0}^\infty
r_Q(n)q^n=\sum_{m=1}^{13}y_m\mathfrak{k}_m(\tau)\quad\textrm{for
some}~ y_1,\cdots,y_{13}\in\mathbb{C}.
\end{equation}
If we set
\begin{equation*}
\mathfrak{k}_m(\tau)=\sum_{n=0}^\infty
c_{n,m}q^n\quad\textrm{for}~m=1,2,\cdots,13,
\end{equation*}
then the relation (\ref{equation}) can be rewritten as
\begin{equation*}
r_Q(n)=\sum_{m=1}^{13}c_{n,m}y_m\quad\textrm{for}~n\geq0.
\end{equation*}
In particular, from the above relations we have the linear system
for $n=0,1,\cdots, 12$
\begin{equation*}
\begin{pmatrix}c_{0,1}&c_{0,2}&\cdots&c_{0,13}\\
c_{1,1}&c_{1,2}&\cdots&c_{1,13}\\&&\vdots&\\c_{12,1}&c_{12,2}&\cdots&c_{12,13}
\end{pmatrix}\begin{pmatrix}y_1\\y_2\\\vdots\\y_{13}\end{pmatrix}=\begin{pmatrix}r_Q(0)\\r_Q(1)\\\vdots\\r_Q(12)\end{pmatrix}.
\end{equation*}
So, by using the Table \ref{Table} and (\ref{r_Q}) we are able to
determine
\begin{eqnarray*}
\Theta_Q(\tau)&=&-\mathfrak{k}_1(\tau)-10\mathfrak{k}_2(\tau)-19\mathfrak{k}_3(\tau)+9\mathfrak{k}_4(\tau)
-24\mathfrak{k}_5(\tau)+36\mathfrak{k}_6(\tau)
+21\mathfrak{k}_7(\tau)\\
&&-37\mathfrak{k}_8(\tau)+35\mathfrak{k}_9(\tau)-9\mathfrak{k}_{10}(\tau)-17\mathfrak{k}_{11}(\tau)
-\mathfrak{k}_{12}(\tau)+18\mathfrak{k}_{13}(\tau).
\end{eqnarray*}
Then, by the product expansion formula (\ref{KleinForm}) we can
easily get the Fourier expansion of $\Theta_Q(\tau)$ as follows:
{\tiny\begin{eqnarray*}
&&1+12q+14q^2+48q^3+36q^4+56q^5+56q^6+84q^7+70q^8+156q^9+48q^{10}+140q^{11}+144q^{12}+168q^{13}+72q^{14}+224q^{15}+132q^{16}\\
&&+216q^{17}+182q^{18}+252q^{19}+168q^{20}+336q^{21}+120q^{22}+288q^{23}+280q^{24}+252q^{25}+170q^{26}+480q^{27}+252q^{28}+360q^{29}+192q^{30}\\
&&+420q^{31}+294q^{32}+560q^{33}+252q^{34}+288q^{35}+468q^{36}+504q^{37}+216q^{38}+672q^{39}+240q^{40}+560q^{41}+288q^{42}+528q^{43}+420q^{44}\\
&&+728q^{45}+336q^{46}+644q^{47}+528q^{48}+516q^{49}+294q^{50}+864q^{51}+504q^{52}+648q^{53}+560q^{54}+480q^{55}+360q^{56}+1008q^{57}+420q^{58}\\
&&+812q^{59}+672q^{60}+744q^{61}+360q^{62}+1092q^{63}+516q^{64}+680q^{65}+480q^{66}+924q^{67}+648q^{68}+1152q^{69}+336q^{70}+980q^{71}+910q^{72}\\
&&+1008q^{73}+432q^{74}+1008q^{75}+756q^{76}+720q^{77}+680q^{78}+960q^{79}+616q^{80}+1452q^{81}+480q^{82}+1148q^{83}+1008q^{84}+1008q^{85}\\
&&+616q^{86}+1440q^{87}+600q^{88}+1232q^{89}+624q^{90}+1020q^{91}+864q^{92}+1680q^{93}+552q^{94}+864q^{95}+1176q^{96}+1344q^{97}+602q^{98}\\
&&+1820q^{99}+756q^{100}+1224q^{101}+1008q^{102}+1248q^{103}+850q^{104}+1152q^{105}+756q^{106}+1296q^{107}+1440q^{108}+1512q^{109}+560q^{110}\\
&&+2016q^{111}+924q^{112}+1368q^{113}+864q^{114}+1344q^{115}+1080q^{116}+2184q^{117}+696q^{118}+1512q^{119}+960q^{120}+1332q^{121}+868q^{122}\\
&&+2240q^{123}+1260q^{124}+1456q^{125}+936q^{126}+1536q^{127}+1190q^{128}+2112q^{129}+672q^{130}+1584q^{131}+1680q^{132}+1296q^{133}\\
&&+792q^{134}+2240q^{135}+1260q^{136}+1904q^{137}+1344q^{138}+1680q^{139}+864q^{140}+2576q^{141}+840q^{142}+1700q^{143}+1716q^{144}\\
&&+1680q^{145}+864q^{146}+2064q^{147}+1512q^{148}+2072q^{149}++1176q^{150}+2100q^{151}+1080q^{152}+2808q^{153}+840q^{154}+1440q^{155}\\
&&+2016q^{156}+1896q^{157}+1120q^{158}+2592q^{159}+1008q^{160}+2016q^{161}+1694q^{162}+2268q^{163}+1680q^{164}+1920q^{165}+984q^{166}\\
&&+2324q^{167}+1440q^{168}+2196q^{169}+864q^{170}+3276q^{171}+1584q^{172}+2088q^{173}+1680q^{174}+1764q^{175}+1540q^{176}+3248q^{177}\\
&&+1056q^{178}+2160q^{179}+2184q^{180}+2184q^{181}+1008q^{182}+2976q^{183}+1680q^{184}+1728q^{185}+1440q^{186}+2520q^{187}+1932q^{188}\\
&&+3360q^{189}+1008q^{190}+2304q^{191}+2064q^{192}+2688q^{193}+1152q^{194}+2720q^{195}+1548q^{196}+2744q^{197}+1560q^{198}+2400q^{199}\\
&&+1470q^{200}+3696q^{201}+1428q^{202}+2520q^{203}+2592q^{204}+1920q^{205}+1456q^{206}+3744q^{207}+1848q^{208}+2160q^{209}+1344q^{210}\\
&&+2544q^{211}+1944q^{212}+3920q^{213}+1512q^{214}+2464q^{215}+2800q^{216}+2160q^{217}+1296q^{218}+4032q^{219}+1440q^{220}+3024q^{221}\\
&&+1728q^{222}+3108q^{223}+1512q^{224}+3276q^{225}+1596q^{226}+3164q^{227}+3024q^{228}+3192q^{229}+1152q^{230}+2880q^{231}+2100q^{232}\\
&&+2808q^{233}+2210q^{234}+2208q^{235}+2436q^{236}+3840q^{237}+1296q^{238}+3332q^{239}+2464q^{240}+3360q^{241}+1554q^{242}+4368q^{243}\\
&&+2232q^{244}+2408q^{245}+1920q^{246}+3060q^{247}+1800q^{248}+4592q^{249}+1248q^{250}+3024q^{251}+3276q^{252}+3360q^{253}+1792q^{254}
\end{eqnarray*}
\begin{eqnarray*}
&&+4032q^{255}+2052q^{256}+3096q^{257}+2464q^{258}+2592q^{259}+2040q^{260}+4680q^{261}+1848q^{262}+3168q^{263}+2400q^{264}+3024q^{265}\\
&&+1512q^{266}+4928q^{267}+2772q^{268}+3240q^{269}+1920q^{270}+3780q^{271}+2376q^{272}+4080q^{273}+1632q^{274}+2940q^{275}+3456q^{276}\\
&&+3336q^{277}+1960q^{278}+5460q^{279}+1680q^{280}+3920q^{281}+2208q^{282}+3408q^{283}+2940q^{284}+3456q^{285}+1680q^{286}+2880q^{287}\\
&&+3822q^{288}+3684q^{289}+1440q^{290}+5376q^{291}+3024q^{292}+4088q^{293}+2408q^{294}+2784q^{295}+2160q^{296}+5600q^{297}+1776q^{298}\\
&&+4032q^{299}+3024q^{300}+3696q^{301}+1800q^{302}+4896q^{303}+2772q^{304}+3472q^{305}+3276q^{306}+4284q^{307}+2160q^{308}+4992q^{309}\\
&&+1680q^{310}+3744q^{311}+3400q^{312}+3768q^{313}+2212q^{314}+3744q^{315}+2880q^{316}+4424q^{317}+3024q^{318}+4200q^{319}+2408q^{320}\\
&&+5184q^{321}+1728q^{322}+4536q^{323}+4356q^{324}+3528q^{325}+1944q^{326}+6048q^{327}+2400q^{328}+3312q^{329}+2240q^{330}+4620q^{331}\\
&&+3444q^{332}+6552q^{333}+1992q^{334}+3168q^{335}+3696q^{336}+4056q^{337}+2198q^{338}+5472q^{339}+3024q^{340}+3600q^{341}+2808q^{342}\\
&&+4200q^{343}+3080q^{344}+5376q^{345}+2436q^{346}+4176q^{347}+4320q^{348}+4872q^{349}+1512q^{350}+\cdots.
\end{eqnarray*}}
\par
From this expansion we happen to numerically find some interesting
identities, which will be conditionally proved in $\S$\ref{sec4},
and so we pose it as a question for a moment:
\begin{equation}\label{conjecture}
r_Q(p^2n)=\frac{r_Q(p^2)r_Q(n)}{r_Q(1)}\quad\textrm{for any prime
$p\neq13$ and any integer $n\geq1$ prime to $p$}.
\end{equation}
Suppose that (\ref{conjecture}) is true. Let $\ell\geq2$ be a
square-free integer which is not divisible by $13$ and has the prime
factorization $\ell=p_1\cdots p_m$. If $n$ is a positive integer
prime to $\ell$, then we derive that
\begin{eqnarray*}
r_Q(\ell^2n)&=&\frac{r_Q(p_1^2)r_Q(p_2^2\cdots p_m^2n)}{r_Q(1)}
=\cdots=\frac{r_Q(p_1^2)\cdots r_Q(p_m^2)r_Q(n)}{r_Q(1)^m}\\
&=&\frac{r_Q(p_1^2p_2^2)r_Q(p_3^2)\cdots
r_Q(p_m^2)r_Q(n)}{r_Q(1)^{m-1}}=\cdots= \frac{r_Q(p_1^2\cdots
p_m^2)r_Q(n)}{r_Q(1)}=\frac{r_Q(\ell^2)r_Q(n)}{r_Q(1)}.
\end{eqnarray*}
Hence we can allow $p$ to be a square-free positive integer not
divisible by $13$ in the question (\ref{conjecture}).
\par
However, unfortunately, the general relation
$r_Q(mn)=r_Q(m)r_Q(n)/r_Q(1)$ for relatively prime positive integers
$m$ and $n$ does not hold because $r_Q(2\cdot 5)=48$ and
$r_Q(2)r_Q(5)/r_Q(1)=196/3$.
\end{example}

\begin{example}\label{examplehalf}
By Remark \ref{firstremark} any product of Klein forms is of
integral weight. So we cannot express $\eta(\tau)$ in terms of Klein
forms. However, as described in \cite{K-L} Chapter 3 Lemma 5.1 we
have the relation
\begin{equation*}
\eta(\tau)^2=\sqrt{\frac{2}{3}}(1-i)\mathfrak{k}_{(\frac{1}{2},0)}(\tau)
\mathfrak{k}_{(0,\frac{1}{2})}(\tau)
\mathfrak{k}_{(\frac{1}{2},\frac{1}{2})}(\tau)
\bigg(\mathfrak{k}_{(\frac{1}{3},0)}(\tau)
\mathfrak{k}_{(0,\frac{1}{3})}(\tau)
\mathfrak{k}_{(\frac{1}{3},\frac{1}{3})}(\tau)
\mathfrak{k}_{(\frac{1}{3},-\frac{1}{3})}(\tau)\bigg)^{-1}.
\end{equation*}
\par
Let $k\geq0$ be an integer and $M_{k/2}(\Gamma_0(4))$ stand for the
space of modular forms for $\Gamma_0(4)$ of weight $k/2$
(\cite{Koblitz} Chapter IV $\S$1). Let
\begin{equation*}
F(\tau)=\sum_{n\geq1~\textrm{odd}}\bigg(\sum_{d>0,~d|n}d\bigg)q^n=q+4q^3+6q^5+8q^7+\cdots.
\end{equation*}
Then one can assign the weight $1/2$ to
$\Theta(\tau)=\eta(2\tau)^5/\eta(\tau)^2\eta(4\tau)^2$ and $2$ to
$F(\tau)$, respectively. And, as is well-known
$M_{k/2}(\Gamma_0(4))$ is the space of all polynomials in
$\mathbb{C}[\Theta(\tau),F(\tau)]$ having pure weight $k/2$
(\cite{Koblitz} Chapter IV Proposition 4). By Theorem
\ref{modularity2} we see that the functions
\begin{eqnarray*}
\mathfrak{k}_{(\frac{2}{4},0)}(4\tau)^{-2}&=&q+4q^3+6q^5+8q^7+\cdots,\\
\mathfrak{k}_{(\frac{1}{4},0)}(4\tau)^{-8}
\mathfrak{k}_{(\frac{2}{4},0)}(4\tau)^{6}&=&1+8q+24q^2+32q^3+24q^4+48q^5+96q^6+64q^7+\cdots
\end{eqnarray*}
belong to $M_2(\Gamma_1(4))=M_2(\Gamma_0(4))$ (Proposition
\ref{decomposition}) which is of dimension $2$ (\cite{Ono} Theorem
1.49). Thus they form a basis of $M_2(\Gamma_0(4))$, from which we
get the following infinite product expansion
\begin{equation*}
F(\tau)=\mathfrak{k}_{(\frac{2}{4},0)}(4\tau)^{-2}
=q\prod_{n=1}^\infty\bigg(\frac{1-q^{4n}}{1-q^{4n-2}}\bigg)^4
=\frac{\eta(4\tau)^8}{\eta(2\tau)^4}.
\end{equation*}
\end{example}

\section{Hecke operators}\label{sec4}

Throughout this section by $\Theta_Q(\tau)=\sum_{n=0}^\infty
r_Q(n)q^n$ we mean the theta function associated with the quadratic
form $Q=x_1^2+2x_2^2+x_3^2+x_4^2+x_1x_3+x_1x_4+x_2x_4$ studied in
Example \ref{fourvariables}. We shall answer the question raised in
(\ref{conjecture}) by making use of Hecke operators on
$\Theta_Q(\tau)$.
\par
Let $N\geq1$ and $k$ be integers, and let $f(\tau)=\sum_{n=0}^\infty
a(n)q^n\in M_k(\Gamma_0(N),\chi)$ for a Dirichlet character $\chi$
modulo $N$. For a positive integer $m$, the \textit{Hecke operator}
$T_{m,k,\chi}$ on $f(\tau)$ is defined by
\begin{equation}\label{Hecke}
f(\tau)|T_{m,k,\chi}=\sum_{n=0}^\infty\bigg(\sum_{d>0,d|\gcd(m,n)}\chi(d)d^{k-1}a(mn/d^2)\bigg)q^n.
\end{equation}
Here we set $\chi(d)=0$ if $\gcd(N,d)\neq1$. As is well-known, the
operator $T_{m,k,\chi}$ preserves the space $M_k(\Gamma_0(N),\chi)$
(\cite{Koblitz} Propositions 36 and 39).
\par
From now on, we let $\chi$ be the Dirichlet character defined by
\begin{equation*}
\chi(d)=\bigg(\frac{13}{d}\bigg)\quad\textrm{for}~d\in\mathbb{Z}-13\mathbb{Z}.
\end{equation*}

\begin{lemma}
The functions $\Theta_Q(\tau)$ and $\Theta_Q(\tau)|T_{13,2,\chi}$
form a basis of $M_2(\Gamma_0(13),\chi)$ over $\mathbb{C}$.
\end{lemma}
\begin{proof}
Note that $M_2(\Gamma_0(13),\chi)$ is of dimension $2$ over
$\mathbb{C}$ (\cite{Ono} Theorem 1.34 and Remark 1.35). We see from
(\ref{r_Q}) and the definition of Hecke operator that
\begin{eqnarray*}
\Theta_Q(\tau)&=&1+12q+14q^2+\cdots\\
\Theta_Q(\tau)|T_{13,2,\chi}&=&1+168q+170q^2+\cdots.
\end{eqnarray*}
Since they are linearly independent over $\mathbb{C}$, these form a
basis of $M_2(\Gamma_0(13),\chi)$.
\end{proof}

\begin{remark}\label{rmk}
If $f(\tau)=\sum_{n=0}^\infty a(n)q^n\in M_2(\Gamma_0(13),\chi)$,
then it can be written as
$c_1\Theta_Q(\tau)+c_2\Theta_Q(\tau)|T_{13,2,\chi}$ for some
$c_1,c_2\in\mathbb{C}$. Since $\begin{pmatrix}1&1\\12&168
\end{pmatrix}$ is invertible, $c_1$ and $c_2$ can be determined only by $a(0)$ and $a(1)$.
In particular, $a(1)=12a(0)=r_Q(1)a(0)$ if and only if
$f(\tau)=a(0)\Theta_Q(\tau)$.
\end{remark}

\begin{proposition}\label{proof}
If $p$ is a prime which satisfies
\begin{eqnarray}
r_Q(p)&=&r_Q(1)(1+\chi(p)p)\quad\textrm{or}\label{primecondition1}\\
r_Q(p^2)&=&r_Q(1)(1+\chi(p)p+p^2)\label{primecondition2},
\end{eqnarray} then
\begin{equation*}
r_Q(p^2n)=\frac{r_Q(p^2)r_Q(n)}{r_Q(1)}\quad\textrm{for any integer
$n\geq1$ prime to $p$}.
\end{equation*}
\end{proposition}
\begin{proof}
Let $p$ be such a prime. We get from the definition (\ref{Hecke})
and the fact $r_Q(0)=1$ that
\begin{eqnarray*}
\Theta_Q(\tau)|T_{p,2,\chi}&=&(1+\chi(p)p)+r_Q(p)q+\cdots\\
\Theta_Q(\tau)|T_{p^2,2,\chi}&=&(1+\chi(p)p+p^2)+r_Q(p^2)q+\cdots.
\end{eqnarray*}
By Remark \ref{rmk} we deduce the assertions
\begin{eqnarray}
r_Q(p)=r_Q(1)(1+\chi(p)p)&\Longleftrightarrow&
\Theta_Q(\tau)|T_{p,2,\chi}=(1+\chi(p)p)\Theta_Q(\tau)\label{derive1}\\
r_Q(p^2)=r_Q(1)(1+\chi(p)p+p^2)&\Longleftrightarrow&
\Theta_Q(\tau)|T_{p^2,2,\chi}=(1+\chi(p)p+p^2)\Theta_Q(\tau).\label{derive2}
\end{eqnarray}
\par
First, suppose that $p$ satisfies (\ref{primecondition1}). For any
integer $n\geq1$ which is prime to $p$, we obtain from the
definition (\ref{Hecke}) and (\ref{derive1}) that
\begin{equation*}
r_Q(p^2)+\chi(p)pr_Q(1)=(1+\chi(p)p)r_Q(p)\quad\textrm{by comparing
the coefficients of $q^p$}.
\end{equation*}
Thus we derive by (\ref{primecondition1}) that
\begin{equation*}
r_Q(p^2)=r_Q(1)(1+\chi(p)p+p^2),
\end{equation*}
which becomes the condition (\ref{primecondition2}).
\par
So we may assume that $p$ satisfies (\ref{primecondition2}). Now,
for $n\geq1$ prime to $p$, we achieve from (\ref{Hecke}) and
(\ref{derive2}) that
\begin{eqnarray*}
r_Q(p^2n)&=&(1+\chi(p)p+p^2)r_Q(n)\quad\textrm{by comparing the
coefficients of $q^n$}\\
&=&r_Q(p^2)r_Q(n)/r_Q(1)\quad\textrm{by (\ref{primecondition2})}.
\end{eqnarray*}
This completes the proof.
\end{proof}

\begin{remark}\label{lastrmk}
\begin{itemize}
\item[(i)] If $p\neq13$ is a prime satisfying (\ref{primecondition1}),
then $\chi(p)$ should be $1$ because $r_Q(p)\geq0$.
\item[(ii)]
By using the explicit Fourier expansion of $\Theta_Q(\tau)$ given in
Example \ref{fourvariables} we can find small primes satisfying the
condition (\ref{primecondition1}) or (\ref{primecondition2}). For
example, $p=3$, $17$, $23$, $29$, $43$, $53$, $61$, $79$, $101$,
$103$, $107$, $113$, $127$, $131$, $139$, $157$, $173$, $179$,
$181$, $191$, $199$, $233$, $251$, $247$, $263$, $269$, $277$,
$283$, $311$, $313$, $337$, $347$ satisfy (\ref{primecondition1}).
And, $p=2$, $5$, $7$, $11$ satisfy (\ref{primecondition2}).
\item[(iii)] We predict that every prime $p\neq13$ satisfies (\ref{primecondition2}).
\item[(iv)] If a prime $p$ satisfies (\ref{primecondition1}) or (\ref{primecondition2}), then
one can easily find a formula for $r_Q(p^n)$ for $n\geq1$. For
example, $p=3$ satisfies (\ref{primecondition1}). It follows from
(\ref{derive1}) that
\begin{equation*}
\Theta_Q(\tau)|T_{3,2,\chi}=(1+\chi(3)3)\Theta_Q(\tau)=4\Theta_Q(\tau).
\end{equation*}
Comparing the coefficients of the term $q^{3^{n-1}}$ ($n\geq2$) on
both sides we have
\begin{equation*}
r_Q(3^{n})+3r_Q(3^{n-2})=4r_Q(3^{n-1}),
\end{equation*}
which can be rewritten as
\begin{equation*}
r_Q(3^{n})-r_Q(3^{n-1})=3\big(r_Q(3^{n-1})-r_Q(3^{n-2})\big).
\end{equation*}
Hence we conclude that
\begin{equation*}
r_Q(3^{n})=r_Q(3^1)+\big(r_Q(3^1)-r_Q(3^0)\big)\sum_{j=1}^{n-1}3^j
=6(3^{n+1}-1)\quad\textrm{for}~n\geq2.
\end{equation*}
Observe that this formula is also true for $n=0$ and $1$.
\item[(v)] Let $A$ be any $4\times4$ positive definite symmetric matrix over
$\mathbb{Z}$ with even diagonal entries. Suppose further that
$\det(A)A^{-1}$ has even diagonal entries. Then, one can apply
Proposition \ref{proof} to the coefficients of the theta function
associated with $A$, if the space
$M_2(\Gamma_0(\det(A)),(\frac{\det(A)}{\cdot}))$ is of dimension
$2$. Unfortunately, it seems that there are no general results on
the construction of a basis of the space $M_2(\Gamma_1(\det(A)))$
(by using products of Klein forms).
\end{itemize}
\end{remark}

\bibliographystyle{amsplain}

\end{document}